\documentclass[12pt]{article}
\usepackage{latexsym,amssymb,amsmath,amsthm,amsfonts}
\usepackage{graphicx,enumerate}
\usepackage{color}
\usepackage{hyperref}

\newtheorem{lemma}{Lemma}[section]

\newtheorem{theorem}[lemma]{Theorem}

\newtheorem{example}[lemma]{Example}

\numberwithin{equation}{section}

\newcommand{\ud}{\,\mathrm{d}}
\newcommand{\RR}{\mathbb{R}}
\newcommand{\f}{\frac}

\newcommand{\xx}{|x|^2}
\newcommand{\yy}{|y|^2}
\newcommand{\xy}{\langle x,y\rangle}

\newcommand{\pp}[2]{\frac{\partial{#1}}{\partial{#2}}}
\newcommand{\ppp}[3]{\frac{\partial^2{#1}}{\partial{#2}\partial{#3}}}
\newcommand{\pppp}[4]%
  {\frac{\partial^3{#1}}{\partial{#2}\partial{#3}\partial{#4}}}

\newcommand{\p}{\phi}
\newcommand{\ps}{\phi(s)}
\newcommand{\pab}{\alpha\phi(\frac{\beta}{\alpha})}

\renewcommand{\a}{\alpha}
\renewcommand{\b}{\beta}
\newcommand{\ab}{(\alpha,\beta)}
\newcommand{\ta}{\tilde\alpha}
\newcommand{\tb}{\tilde\beta}
\newcommand{\ha}{\hat\alpha}
\newcommand{\hb}{\hat\beta}
\newcommand{\ba}{\bar\alpha}
\newcommand{\bb}{\bar\beta}

\newcommand{\aij}{a_{ij}}

\newcommand{\bi}{b_i}
\newcommand{\bj}{b_j}
\newcommand{\bk}{b_k}

\newcommand{\hbi}{\hat b_i}
\newcommand{\hbj}{\hat b_j}

\newcommand{\bbi}{\bar b_i}
\newcommand{\bbj}{\bar b_j}

\newcommand{\bij}{b_{i|j}}
\newcommand{\taij}{\tilde a_{ij}}

\newcommand{\tbij}{\tilde b_{i|j}}
\newcommand{\haij}{\hat a_{ij}}

\newcommand{\hbij}{\hat b_{i|j}}
\newcommand{\baij}{\bar a_{ij}}

\newcommand{\bbij}{\bar b_{i|j}}

\newcommand{\G}{G^i_\alpha}
\newcommand{\tG}{\tilde G^i_{\tilde\alpha}}
\newcommand{\hG}{\hat G^i_{\hat\alpha}}
\newcommand{\bG}{\bar G^i_{\bar\alpha}}

\newcommand{\rij}{r_{ij}}
\newcommand{\sij}{s_{ij}}
\newcommand{\ri}{r_i}
\newcommand{\si}{s_i}
\newcommand{\sio}{s^i{}_0}
\newcommand{\rj}{r_j}
\newcommand{\sj}{s_j}
\newcommand{\trij}{\tilde r_{ij}}
\newcommand{\hrij}{\hat r_{ij}}
\newcommand{\brij}{\bar r_{ij}}

\begin{document}
\title{On dually flat $\ab$-metrics
}
\author{Changtao Yu\footnote{supported by a NSFC grant(No.11026097)}}
\date{}
\maketitle

\begin{abstract}
The dual flatness for Riemannian metrics in information geometry has been extended to Finsler metrics. The aim of this paper is to study the dual flatness of the so-called $\ab$-metrics in Finsler geometry. By doing some special deformations, we will show that the dual flatness of an $\ab$-metric always arises from that of some Riemannian metric in dimensional $n\geq3$.
\end{abstract}

\section{Introduction}
Dual flatness is a basic notion in information geometry. It was first proposed by S.-I. Amari and H. Nagaoka when they study the information geometry on Riemannian spaces\cite{AN2}. Information geometry has emerged from investigating the geometrical structure of a family of probability distributions, and has been applied successfully to
various areas including statistical inference, control system theory and multiterminal information theory\cite{AN1,AN2}.

In 2007, Z. Shen extended the dual flatness in Finsler geometry\cite{szm-rfgw}. A Finsler metric $F$  on a manifold $M$ is said to be {\it locally dually flat} if at any point there is a local coordinate system $(x^i)$ in which $F=F(x,y)$ satisfies the following PDEs
\begin{eqnarray*}
[F^2]_{x^ky^l}y^k-2[F^2]_{x^l}=0.
\end{eqnarray*}
Such a coordinate system is said to be {\it adapted}.

For a Riemannian metric $\a=\sqrt{a_{ij}(x)y^iy^j}$, it is known that $\a$ is locally dually flat if and only if in an adapted coordinate system, the fundamental tensor of $\a$ is the Hessian of some local smooth function $\psi(x)$\cite{AN1,AN2}, i.e.,
\begin{eqnarray*}
a_{ij}(x)=\ppp{\psi}{x^i}{x^j}(x).
\end{eqnarray*}
The dual flatness of a Riemannian metric can also be described by its spray\cite{yct-odfr}: {\it $\a$ is locally dually flat if and only if its spray coefficients could be expressed in an adapted coordinate system} as
\begin{eqnarray}\label{duallyflat}
\G=2\theta y^i+\a^2\theta^i
\end{eqnarray}
{\it for some $1$-form $\xi:=\xi_i(x)y^i$}.

The first example of non-Riemannian dually flat Finsler metrics is the co-call {\it Funk metric}
\begin{eqnarray*}
F=\f{\sqrt{(1-\xx)\yy+\xy^2}}{1-\xx}+\f{\xy}{1-\xx}
\end{eqnarray*}
on the unit ball $\mathbb B^n(1)$\cite{csz-oldf}, which belongs to a special class of Finsler metrics named {\it Randers metrics}. Randers metrics are expressed as the sum of a Riemannian metric $\a=\sqrt{a_{ij}(x)y^iy^j}$ and an $1$-form $\b=b_i(x)y^i$ with the norm $b:=\|\b\|_\a<1$.

Based on the characterization result for locally dually flat Randers metrics given by X. Cheng et al.\cite{csz-oldf}, the author provide a more direct characterization and prove that the dual flatness of a Randers metric always arises from that of some Riemannian metric\cite{yct-odfr}: {\it A Randers metric $F=\a+\b$ is locally dually flat if and only if the Riemannian metric $\ba=\sqrt{1-b^2}\sqrt{\a^2-\b^2}$ is locally dually flat and the $1$-form $\bb=-(1-b^2)\b$ is dually related with respect to $\ba$}. In this case, $F$ can be reexpressed as
\begin{eqnarray}\label{NPE}
F=\f{\sqrt{(1-\bar b^2)\ba^2+\bb^2}}{1-\bar b^2}-\f{\bb}{1-\bar b^2}.
\end{eqnarray}

Recall that an $1$-form $\b$ is said to be {\it dually related} to a locally dually flat Riemannian metric $\a$ if in an adopted coordinate system the spray coefficients of $\a$ are in the form (\ref{duallyflat}) and the covariant derivation of $\b$ with respect to $\a$ are given by
\begin{eqnarray}\label{duallyrelated}
\bij=2\theta_i\bj+c(x)\aij
\end{eqnarray}
for some scalar function $c(x)$. This concept was first introduced by the author in \cite{yct-odfr}. In particular, we prove that the Riemannian metrics
\begin{eqnarray}\label{dfR}
\ba=\frac{\sqrt{(1+\mu|x|^2)|y|^2-\mu\langle x,y\rangle^2}}{(1+\mu|x|^2)^\frac{3}{4}}
\end{eqnarray}
are dually flat on the ball $\mathbb B^n(r_\mu)$ and the $1$-forms
\begin{eqnarray}\label{drb}
\bb=\f{\lambda\xy}{(1+\mu|x|^2)^\frac{5}{4}}
\end{eqnarray}
are dually related to $\ba$ for any constant number $\mu$ and $\lambda$, where the the radius $r_\mu$ is given by $r_\mu=\frac{1}{\sqrt{-\mu}}$ if
$\mu<0$ and $r_\mu=+\infty$ if $\mu\geq0$.

As a result, we construct many non-trivial dually flat Randers metrics as following:
\begin{eqnarray*}
F(x,y)&=&\f{\sqrt[4]{1+(\mu+\lambda^2)\xx}\sqrt{(1+\mu\xx)\yy-\mu\xy^2}}{1+\mu|x|^2}\nonumber\\
&&+\f{\lambda\xy}{(1+\mu\xx)\sqrt[4]{1+(\mu+\lambda^2)\xx}}.
\end{eqnarray*}
It is just the Funk metric when $\mu=-1$ and $\lambda=1$.

(\ref{NPE}) is just the {\it navigation expression} for Randers metrics, which play a key role in the research of Randers metrics. For example, D. Bao et al. classified Randers metrics with constant flag curvature \cite{brs}:~{\it $F=\a+\b$ is of constant flag curvature if and only if $\ba$ in (\ref{NPE}) is of constant sectional curvature and $\bb$ is homothetic to $\ba$, i.e.,}
$$\frac{1}{2}\left(\bar b_{i|j}+\bar b_{j|i}\right)=c\bar a_{ij}$$
{\it for some constant $c$.} Similarly, D. Bao et al. gave a characterization for Einstein metric of Randers type\cite{db-robl-onri}:~{\it $F=\a+\b$ is Einsteinian if and only if $\ba$ is Einsteinian and $\bb$ is homothetic to $\ba$}. It seems that most of the properties of Randers metrics become simple and clear if they are described with the navigation form\cite{huang}.

Except for Randers metrics, there is another important class of Finsler metrics defined also by a Riemannian metric and an $1$-form and given in the form
$$F=\pab,$$
where $\ps$ is a smooth function. Such kinds of Finsler metrics are called {\it $\ab$-metrics}. It was proposed by M. Matsumoto in 1972 as a direct generalization of Randers metrics. $\ab$-metrics form a special class of Finsler metrics partly because of its computability\cite{bcs}. Recently, many encouraging results about $\ab$-metrics, including flag curvature property\cite{lb-szm-onac,zhou}, Ricci curvature property\cite{cxy-szm-tyf,szm-yct-oesm} and projective property \cite{szm-opfa,yct-dhfp} etc., have been achieved.

2011, Q. Xia give a local characterization of locally dually flat $\ab$-metrics on a manifold with dimension $n\geq3$:
\begin{theorem}\label{maincf}\cite{xoldf}
Let $F=\pab$ be a Finsler metric on an open subset $U\subseteq\RR^n$ with $n\geq3$. Suppose $F$ is not of Riemannian type and $\phi'(0)\neq0$. Then $F$ is dually flat on $U$ if and only if the following conditions hold:
\begin{eqnarray}
&&\G=[2\theta+(3k_1-2)\tau\b]y^i+\a^2(\theta^i-\tau b^i)+\f{3}{2}k_3\tau\b^2 b^i,\label{Gi}\\
&&r_{00}=2\theta\b+(3\tau+2\tau b^2-2b_k\theta^k)\a^2+(3k_2-2-3k_3b^2)\tau\b^2,\label{rij}\\
&&s_{i0}=\b\theta_i-\theta\bi,\label{sij}\\
&&\tau\left\{s(k_2-k_3s^2)(\p\p'-s\p'^2-s\p\p'')-(\p'^2+\p\p'')+k_1\p(\p-s\p')\right\}=0,~~~~
\end{eqnarray}
where $\theta$ is an $1$-form, $\tau$ is a scalar function, and $k_1,k_2,k_3$ are constants.
\end{theorem}
The meaning of some notations here can be found in Section 2.

When $\tau=0$, (\ref{Gi}) becomes $\G=2\theta y^i+\a^2\theta^i$, which implies $\a$ is dually flat. Moreover, (\ref{rij}) and (\ref{sij}) are equivalent to $\bij=2\theta_i\bj-2b_k\theta^k\aij$, i.e., $\b$ is dually related to $\a$ with $c(x)+2b_k\theta^k=0$. In fact, this is a {\it trivial} case. Because in this case, $F=\pab$ will be always dually flat for any suitable function $\ps$ by Theorem \ref{maincf}. In this paper, we will focus on the non-trivial case. Thereby, the function $\ps$ must satisfy a 3-parameters equation
\begin{eqnarray}\label{phi}
s(k_2-k_3s^2)(\p\p'-s\p'^2-s\p\p'')-(\p'^2+\p\p'')+k_1\p(\p-s\p')=0.
\end{eqnarray}

It is clear that the geometry meaning of the original data $\a$ and $\b$ for the dually flat $\ab$-metrics is very obscure. The main aim of this paper is to provide a luminous description for a non-trivial dually flat $\ab$-metric. By using a special class of metric deformations called {\it $\b$-deformations}, we prove that {\it the dual flatness of an $\ab$-metrics always arises from that of some Riemannian metric}, just as Randers metrics.
\begin{theorem}\label{main1}
Let $F=\pab$ be a Finsler metric on an open subset $U\subseteq\RR^n$ with $n\geq3$, where $\ps$ satisfies (\ref{phi}). Suppose $F$ is not of Riemannian type and $\phi'(0)\neq0$. Then $F$ is dually flat if and only if $\a$ and $\b$ can be expressed as
\begin{eqnarray*}
\a=\eta(\bar b^2)\sqrt{\ba^2-\f{(k_2-k_3\bar
b^2)}{1+k_2\bar b^2-k_3\bar b^4}\bb^2},\quad\b=-\f{\eta(\bar b^2)}{(1+k_2\bar b^2-k_3\bar b^4)^\f{1}{2}}\bb,
\end{eqnarray*}
where $\ba$ is a dually flat Riemannian metric on $U$, $\bb$ is dually related to $\ba$, $\bar b:=\|\bb\|_{\ba}$.
The deformation factor $\eta(\bar b^2)$ is determined by the coefficients $k_1,k_2,k_3$ and given in the following five case,
\begin{enumerate}[(1)]
\item When $k_3=0,~k_2=0$,
$$\eta(\bar b^2)=\exp\left\{\f{k_1\bar b^2}{4}\right\};$$
\item When $k_3=0,~k_2\neq0$,
$$\eta(\bar b^2)=\left\{1+k_2\bar b^2\right\}^{\f{k_1-k_2}{4k_2}};$$
\item When $k_3\neq0,~\Delta_1>0$,
$$\eta(\bar b^2)=\f{\left\{\f{\sqrt{\Delta_1}+k_2}{\sqrt{\Delta_1-k_2}}\cdot\f{\sqrt\Delta_1-k_2+2k_3\bar b^2}{\sqrt\Delta_1+k_2-2k_3\bar
b^2}\right\}^\f{2k_1-k_2}{8\sqrt\Delta_1}}{\sqrt[8]{1+k_2\bar
b^2-k_3\bar b^4}};$$
\item When $k_3\neq0,~\Delta_1=0$,
$$\eta(\bar b^2)=\f{\sqrt[4]{2}\exp\left\{\f{k_2-2k_1}{2k_2}\left[\f{1}{2+k_2\bar b^2}-\f{1}{2}\right]\right\}}{\sqrt[4]{2+k_2\bar b^2}};$$
\item When $k_3\neq0,~\Delta_1<0$,
$$\eta(\bar b^2)=\f{\exp\left\{\f{2k_1-k_2}{4\sqrt{-\Delta_1}}\left(\arctan\f{k_2-2k_3\bar
b^2}{\sqrt{-\Delta_1}}-\arctan\f{k_2}{\sqrt{-\Delta_1}}\right)\right\}}{\sqrt[8]{1+k_2\bar b^2-k_3\bar b^4}},$$
\end{enumerate}
where $\Delta_1:=k_2^2+4k_3$.
\end{theorem}

$\b$-deformations, which play a key role in the proof of Theorem \ref{main1}, are a new method in Riemann-Finsler geometry developed by the author in the research of projectively flat $\ab$-metrics\cite{yct-dhfp}. Roughly speaking, the method of $\b$-deformations is aim to make clear the
latent light. For an analogy, $\a$ and $\b$ just like two ropes tangles together, and
it is possible to unhitch them using $\b$-deformations. The navigation expression
for Randers metrics is a representative example. In fact, it is just a specific
kind of $\b$-deformations. In other words, $\b$-deformations can be regarded as a
natural generalization of the navigation expression for Randers metrics. See also \cite{szm-yct-oesm} for more applications.

The argument in this paper is similar to that in \cite{yct-odfr}, but we don't show the original ideas here. One can obtain a more deep analysis in the latter.

In Section \ref{4}, we will use a skillful method to solve the basic equation (\ref{phi}). As a result, we can construct infinity many non-trivial dually flat $\ab$-metrics combining with (\ref{dfR}) and (\ref{drb}). In particular, the following metrics
$$F=\sqrt{\alpha^2+2\varepsilon\a\b+\kappa\b^2}$$
is locally dually flat if and only if
\begin{eqnarray}\label{example}
\a=(1-\kappa\bar b^2)^{-1}\sqrt{(1-\kappa\bar b^2)\ba^2+\kappa\bb^2},\qquad\b=-(1-\kappa\bar b^2)^{-1}\bb,
\end{eqnarray}
where $\ba$ is locally dually flat and $\bb$ is dually related to $\ba$.

Taking $\kappa=1$ and $\varepsilon=1$, one can see that (\ref{example}) is just the Randers metrics $F=\a+\b$. Taking $\kappa=0$ and $\varepsilon=\frac{1}{2}$, then we can obtain a very simple kind of dually flat $\ab$-metrics given in the form
$$F=\sqrt{\a(\a+\b)}.$$

\section{Preliminaries}\label{2}
Let $M$ be a smooth $n$-dimensional manifold. A Finsler metric $F$ on $M$ is a continuous function
$F:TM\to[0,+\infty)$ with the following properties:
\begin{enumerate}[(i)]
\item {\it Regularity}: $F$ is $C^\infty$ on the entire slit tangent bundle $TM\backslash\{0\}$;
\item {\it Positive homogeneity}: $F(x,\lambda y)=\lambda F(x,y)$ for all $\lambda>0$;
\item {\it Strong convexity}: the fundamental tensor $g_{ij}:=[\frac{1}{2}F^2]_{y^iy^j}$ is positive definite for all $(x,y)\in TM\backslash\{0\}$.
\end{enumerate}
Here $x=(x^i)$ denotes the coordinates of the point in $M$ and $y=(y^i)$ denotes the coordinates of the vector in $T_xM$.

For a Finsler metric, the {\it geodesics} are characterized by the geodesic equation
$$\ddot c^i(t)+2G^i\left(c(t),\dot c(t)\right)=0,$$
where
$$G^i(x,y):=\f{1}{4}g^{il}\left\{[F^2]_{x^ky^l}y^k-[F^2]_{x^l}\right\}$$
are called the {\it spray coefficients} of $F$. Here $(g^{ij}):=(g_{ij})^{-1}$. For a Riemannian metric $\a$, the spray coefficients are given by
$$G^i_\a(x,y)=\f{1}{2}\Gamma^i{}_{jk}(x)y^jy^k$$
in terms of the Christoffel symbols of $\a$.

By definition, an $\ab$-metric is a Finsler metric in the form $F=\pab$, where $\a=\sqrt{a_{ij}(x)y^iy^j}$ is a Riemannian metric, $\b=b_i(x)y^i$ is an $1$-form and $\ps$ is a positive smooth function on some symmetric open interval $(-b_o,b_o)$.

On the other hand, the so-called $\b$-deformations are a triple of metric deformations in terms of $\a$ and $\b$ listed below:
\begin{eqnarray*}
&\ta=\sqrt{\a^2-\kappa(b^2)\b^2},\qquad\tb=\b;\\
&\ha=e^{\rho(b^2)}\ta,\qquad\hb=\tb;\\
&\ba=\ha,\qquad\bb=\nu(b^2)\hb.
\end{eqnarray*}

Some basic formulas for $\b$-deformations are listed below. Be attention that the notation `$\dot b_{i|j}$' always means the covariant derivative of the $1$-form `$\dot\b$' with respect to the corresponding Riemannian metric `$\dot\a$', where the symbol `~$\dot{}$~' can be nothing, `~$\tilde{}$~', `~$\hat{}$~' or `~$\bar{}$~' in this paper. Moreover, we need the following abbreviations,
\begin{eqnarray*}
&r_{00}:=r_{ij}y^iy^j,~r_i:=r_{ij}y^j,~r_0:=r_iy^i,~r:=r_ib^i,\\
&s_{i0}:=s_{ij}y^j,~s^i{}_0:=a^{ij}s_{j0},~s_i:=s_{ij}y^j,~s_0:=s_ib^i,
\end{eqnarray*}
where $\rij$ and $\sij$ are the symmetrization and antisymmetrization of $\bij$ respectively, i.e.,
$$\rij:=\f{1}{2}(\bij+b_{j|i}),\quad\sij:=\f{1}{2}(\bij-b_{j|i}).$$
Roughly speaking, indices are raised and lowered by $\aij$, vanished by contracted with $b^i$ and changed to be `${}_0$' by contracted with $y^i$. Since $\bij-b_{j|i}=\frac{\partial b_i}{\partial x^j}-\pp{b_j}{x^i}$, $\sij=0$ implies $\b$ is closed, and vice versa.

\begin{lemma}\cite{yct-dhfp}\label{beta1}
Let $\ta=\sqrt{\a^2-\kappa(b^2)\b^2}$, $\tb=\b$. Then
\begin{eqnarray*}
\tG&=&\G-\frac{\kappa}{2(1-\kappa b^2)}\big\{2(1-\kappa b^2)\b s^i{}_0+r_{00}b^i+2\kappa s_0\b b^i\big\}\\
&&+\frac{\kappa'}{2(1-\kappa b^2)}\big\{(1-\kappa b^2)\b^2(r^i+s^i)+\kappa r\b^2b^i-2(r_0+s_0)\b
b^i\big\},\\
\tbij&=&\bij+\frac{\kappa}{1-\kappa b^2}\big\{b^2\rij+\bi\sj+\bj\si\big\}\\
&&-\frac{\kappa'}{1-\kappa b^2}\big\{r\bi\bj-b^2\bi(\rj+\sj)-b^2\bj(\ri+\si)\big\}.
\end{eqnarray*}
\end{lemma}

\begin{lemma}\cite{yct-dhfp}\label{beta2}
Let $\ha=e^{\rho(b^2)}\ta$, $\hb=\tb$. Then
\begin{eqnarray*}
\hG&=&\tG+\rho'\left\{2(r_0+s_0)y^i-(\a^2-\kappa \b^2)\left(r^i+s^i+\frac{\kappa}{1-\kappa b^2}rb^i\right)\right\},\\
\hbij&=&\tbij-2\rho'\left\{\bi(\rj+\sj)+\bj(\ri+\si)-\frac{1}{1-\kappa b^2}r(\aij-\kappa\bi\bj)\right\}.
\end{eqnarray*}
\end{lemma}

\begin{lemma}\cite{yct-dhfp}\label{beta3}
Let $\ba=\ha$, $\bb=\nu(b^2)\hb$. Then
\begin{eqnarray*}
\bG&=&\hG,\\
\bbij&=&\nu\hbij+2\nu'\bi(\rj+\sj).
\end{eqnarray*}
\end{lemma}

\section{Proof of Theorem \ref{main1}}\label{3}
Suppose that $F=\pab$ is a non-trivial dually flat $\ab$-metric on $U$. According to Theorem \ref{maincf}, it is easy to obtain the following simple facts:
\begin{eqnarray}
\rij&=&\theta_i\bj+\theta_j\bi+(3\tau+2\tau b^2-2b_k\theta^k)\aij+\tau(3k_2-2-3k_3b^2)\bi\bj,\label{temp1}\\
\sio&=&\b\theta^i-\theta b^i,\label{temp2}\\
s_0&=&b_k\theta^k\b-b^2\theta,\label{temp3}\\
\ri+\si&=&3\tau(1+k_2b^2-k_3b^4)\bi,\label{temp4}\\
\bi\sj+\bj\si&=&2\bk\theta^k\bi\bj-b^2(\theta_i\bj+\theta_j\bi),\label{temp5}\\
r&=&3\tau(1+k_2b^2-k_3b^4)b^2.\label{temp6}
\end{eqnarray}
\begin{lemma}\label{step1}
Take $\kappa(b^2)=-k_2+k_3b^2$, then
\begin{eqnarray*}
\tG=[2\theta+\tau\b(3k_1-2)]y^i+\ta^2\theta^i
+\f{\tau(3k_2-2-3k_3b^2)-2(k_2-k_3b^2)b_k\theta^k}{2(1+k_2b^2-k_3b^4)}\ta^2b^i.
\end{eqnarray*}
\end{lemma}
\begin{proof}
By (\ref{Gi}), (\ref{temp1})-(\ref{temp6}) and Lemma \ref{beta1}, we have
\begin{eqnarray*}
\tG&=&[2\theta+(3k_1-2)\tau\b]y^i+\a^2(\theta^i-\tau b^i)+\f{3}{2}k_3\tau\b^2 b^i\\
&&-\f{\kappa}{2(1-\kappa b^2)}\big\{2(1-\kappa b^2)\b(\b\theta^i-\theta b^i)+2\theta\b b^i+(3\tau+2\tau b^2-b_k\theta^k)\a^2 b^i\\
&&+\tau(3k_2-2-3k_3b^2)\b^2 b^i+2\kappa(b_k\theta^k\b-b^2\theta)\b b^i\big\}\\
&&+\f{\kappa'}{2(1-\kappa b^2)}\big\{3\tau(1-\kappa b^2)(1+k_2b^2-k_3b^4)\b^2b^i\\
&&+3\tau\kappa(1+k_2b^2-k_3b^4)b^2\b^2b^i-6\tau(1+k_2b^2-k_3b^4)\b^2 b^i\big\}\\
&=&[2\theta+(3k_1-2)\tau\b]y^i+\ta^2\theta^i-\f{1}{2(1-\kappa b^2)}\big\{(3\tau\kappa+2\tau-2\kappa b_k\theta^k)\a^2\\
&&+[2\kappa^2b_k\theta^k-3\tau k_3(1-\kappa b^2)+\tau\kappa(3k_2-2-3k_3b^2)+3\tau\kappa'(1-k_2b^2+k_3b^4)]\b^2\big\}b^i.
\end{eqnarray*}
When $\kappa=-k_2+k_3b^2$, it is easy to verify that
$$\kappa^2+k_2\kappa-k_3=-\kappa'(1+k_2b^2-k_3b^4),$$
and hence $\tG$ is given in the following form,
\begin{eqnarray}\label{tGt}
\tG=[2\theta+\tau\b(3k_1-2)]y^i+\ta^2\theta^i-\f{3\tau\kappa+2\tau-2\kappa b_k\theta^k}{2(1-\kappa b^2)}\ta^2b^i.
\end{eqnarray}
\end{proof}

\begin{lemma}\label{step2}
Take $\rho(b^2)=-\frac{1}{4}\int\f{k_1-k_2+k_3b^2}{1+k_2b^2-k_3b^4}\ud b^2$, then
\begin{eqnarray*}
\hG=2\hat\theta y^i+\ha^2\hat\theta^i,
\end{eqnarray*}
where $\hat\theta=\theta-\frac{1}{4}\tau[4-3(k_1+k_2-k_3b^2)]\b$. In particular, $\ha$ is dually flat on $U$.
\end{lemma}
\begin{proof}
by (\ref{temp4}), (\ref{temp6}), (\ref{tGt}) and Lemma \ref{beta2} we have
\begin{eqnarray*}
\hG&=&\tG+\rho'\Big\{6\tau(1+k_2b^2-k_3b^4)\b y^i-\ta^2\Big(3\tau(1+k_2b^2-k_3b^4)b^i\\
&&+\f{\kappa}{1-\kappa b^2}\cdot3\tau(1+k_2b^2-k_3b^4)b^2b^i\Big)\Big\}\\
&=&\left\{2\theta+\tau[3k_1-2+6\rho'(1+k_2b^2-k_3b^4)]\b\right\}y^i+\ta^2\theta^i\\
&&-\f{1}{2(1-\kappa b^2)}\left\{3\tau\kappa+2\tau+6\tau\rho'(1+k_2b^2-k_3b^4)-2\kappa b_k\theta^k\right\}\ta^2b^i.
\end{eqnarray*}
Let
\begin{eqnarray*}
\hat\theta:=\theta+\f{1}{2}\tau[3k_1-2+6\rho'(1+k_2b^2-k_3b^4)]\b.
\end{eqnarray*}
It is easy to verify that the inverse of $(\haij)$ is given by
\begin{eqnarray}\label{haIJ}
\hat a^{ij}=e^{-2\rho}\left(a^{ij}+\f{\kappa}{1-\kappa b^2}b^ib^j\right),
\end{eqnarray}
so $\hat\theta^i:=\hat a^{ij}\hat\theta_j$ are given by
\begin{eqnarray*}
\hat\theta^i=e^{-2\rho}\left\{\theta^i+\f{1}{2(1-\kappa b^2)}\left[2\kappa b_k\theta^k+\tau(3k_1-2)
+6\tau\rho'(1+k_2b^2-k_3b^4)\right]b^i\right\}.
\end{eqnarray*}
Hence $\hG$ can be reexpressed as
\begin{eqnarray*}
\hG=2\hat\theta y^i+\ha^2\hat\theta^i-\f{3\tau e^{-2\rho}}{2(1-\kappa b^2)}\left\{k_1+\kappa+4\rho'(1+k_2b^2-k_3b^4)\right\}\ha^2b^i.
\end{eqnarray*}
Obviously, the deformation factor given in the Lemma satisfies
\begin{eqnarray}\label{rho}
\rho'=-\f{k_1+\kappa}{4(1+k_2b^2-k_3b^4)},
\end{eqnarray}
thus $\hG=2\hat\theta y^i+\ha^2\hat\theta^i$.
\end{proof}

\begin{lemma}\label{step3}
Take $\nu(b^2)=-\sqrt{1+k_2b^2-k_3b^4}e^{\rho(b^2)}$, then
\begin{eqnarray*}
\bG&=&2\bar\theta y^i+\ba^2\bar\theta^i,\\
\bbij&=&2\bar\theta_i\bbj+\bar c(x)\baij,
\end{eqnarray*}
where $\bar c(x)$ is a scalar function. In particular, $\bb$ is dually related to $\ba$.
\end{lemma}
\begin{proof}
Under the deformations used above, combining with (\ref{temp1}), (\ref{temp4}), (\ref{temp5}) and Lemma \ref{beta2} we can see that
\begin{eqnarray*}
\trij&=&\f{1}{1-\kappa b^2}\big\{\rij+2\kappa b_k\theta^k\bi\bj-\kappa b^2(\theta_i\bj+\theta_j\bi)
+3\tau \kappa'(1+k_2b^2-k_3b^4)b^2\bi\bj\big\}\\
&=&\theta_i\bj+\theta_j\bi+\f{1}{1-\kappa b^2}\big\{(3\tau+2\tau b^2-2b_k\theta^k)\aij\\
&&+[\tau(3k_2-2-3k_3b^2)+2\kappa b_k\theta^k+3\tau \kappa'(1+k_2b^2-k_3b^4)b^2]\bi\bj\big\}\\
&=&\theta_i\bj+\theta_j\bi+\f{1}{1-\kappa b^2}\left(3\tau+2\tau b^2-2\bk\theta^k\right)\taij+\tau(3\kappa+3k_2-2)\bi\bj,\\
\tilde s_{ij}&=&\sij=\theta_i\bj-\theta_j\bi.
\end{eqnarray*}
Similarly, by (\ref{temp4}), (\ref{rho}) and Lemma \ref{beta2} we get
\begin{eqnarray*}
\hrij&=&\trij+\f{k_1+\kappa}{2(1+k_2b^2-k_3b^4)}\Big\{6\tau(1+k_2b^2-k_3b^4)\bi\bj
-\f{1}{1-\kappa b^2}\cdot3\tau(1+k_2b^2-k_3b^4)b^2\taij\Big\}\\
&=&\theta_i\bj+\theta_j\bi+\f{e^{-2\rho}}{2(1-\kappa b^2)}\left\{6\tau+(4-3k_1)\tau b^2-3\tau\kappa b^2-4\bk\theta^k\right\}\haij
+\tau(6\kappa+3k_1+3k_2-2)\bi\bj,\\
\hat s_{ij}&=&\sij=\theta_i\bj-\theta_j\bi.
\end{eqnarray*}
If we use $\hat\theta$ instead of $\theta$ to express $\hrij$ and $\hat s_{ij}$, then
\begin{eqnarray*}
\hrij&=&\hat\theta_i\hbj+\hat\theta_j\hbi+\f{e^{-2\rho}}{2(1-\kappa b^2)}\left\{6\tau+\tau b^2-3\tau\kappa b^2-4\bk\theta^k\right\}\haij\\
&&+\frac{3}{2}\tau(5\kappa+k_1+2k_2)\hbi\hbj,\\
\hat s_{ij}&=&\hat\theta_i\hbj-\hat\theta_j\hbi,
\end{eqnarray*}
where $\hbi=\bi$ according to the definition of $\b$-deformations.

Finally, by (\ref{temp4}) and Lemma \ref{beta3} we have
\begin{eqnarray*}
\brij&=&\nu\hrij+6\tau\nu'(1+k_2b^2-k_3b^4)\bi\bj,\\
&=&\bar\theta_i\bbj+\bar\theta_j\bbi+\f{e^{-2\rho}\nu}{2(1-\kappa b^2)}\left\{6\tau+\tau b^2-3\tau \kappa b^2-4\bk\theta^k\right\}\baij\\
&&+\frac{3}{2}\tau\left\{(5\kappa+k_1+2k_2)\nu+4(1+k_2b^2-k_3b^4)\nu'\right\}\hbi\hbj,\\
\bar s_{ij}&=&\nu\sij=\nu(\hat\theta_i\hbj-\hat\theta_j\hbi)=\bar\theta_i\bbj-\bar\theta_j\bbi,
\end{eqnarray*}
where $\bar\theta:=\hat\theta$. It is easy to verify that the deformation factor in the Lemma satisfies
\begin{eqnarray}
(5\kappa+k_1+2k_2)\nu+4(1+k_2b^2-k_3b^4)\nu'=0,
\end{eqnarray}
So
$$\brij=\bar\theta_i\bbj+\bar\theta_j\bbi+\bar c(x)\baij$$
where $\bar c(x)$ is a scalar function and can be reexpressed as
\begin{eqnarray}\label{barc}
\bar c(x)=-2\bar b_k\bar\theta^k+\f{3\tau e^{-2\rho}\nu}{2(1-\kappa b^2)}\left\{2(1-\kappa b^2)+(k_1-1)b^2\right\}.
\end{eqnarray}
Combining with $\bar s_{ij}$, we have $\bbij=2\bar\theta_i\bbj+\bar c(x)\baij$.
\end{proof}

From the equality (\ref{barc}) we can see that $\bar c(x)\neq-2\bar b_k\bar\theta^k$ unless $\tau=0$. In other words, when $\tau\neq0$, $\bb$ is non-trivial~(see the statements below Theorem \ref{maincf} for the reason).

\begin{proof}[Proof of Theorem \ref{main1}]
Due to the above Lemmas, we have show that if $F=\pab$ is a non-trivial dually flat Finsler metric with dimension $n\geq3$, then the output Riemannian metric $\ba$ is dually flat and the output $1$-form $\bb$ is dually related to $\ba$.

Conversely, by (\ref{haIJ}) we can see that the norm of $\bar b$ is related to $b$ as
$$\bar b^2=\nu b_i\nu b_je^{-2\rho}\left(a^{ij}+\f{\kappa}{1-\kappa b^2}b^ib^j\right)=b^2,$$
which implies that the $\b$-deformations given above are reversible. More specifically, we have
$$\b=\nu^{-1}(\bar b^2)\bb=-\f{e^{-\rho(\bar b^2)}}{\sqrt{1+k_2\bar b^2-k_3\bar b^4}}\bb$$
and
$$\a=\sqrt{e^{-2\rho(\bar b^2)}\ba^2+\kappa(\bar b^2)\b^2}=e^{-\rho(\bar b^2)}\sqrt{\ba^2-\f{(k_2-k_3\bar
b^2)}{1+k_2\bar b^2-k_3\bar b^4}\bb^2}.$$
Denote $\eta(\bar b^2):=e^{-\rho(\bar b^2)}$. By (\ref{rho}), $\eta$ can be chose as
$$\eta(\bar b^2)=\exp{\left\{\f{1}{4}\int_0^{\bar b^2}\f{k_1-k_2+k_3t}{1+k_2t-k_3t^2}\ud t\right\}}.$$

Combining with the discussions in the proofs of Lemma \ref{step1}, Lemma \ref{step2} and Lemma \ref{step3}, it is not hard to see that if $\ba$ is dually flat and $\bb$ is dually related to $\ba$, then the output data $\a$ and $\b$ of the reverse $\b$-deformations satisfy (\ref{Gi})-(\ref{sij}) and hence $F=\pab$ is dually flat.
\end{proof}

\section{Symmetry and solutions of equation (\ref{phi})}\label{4}
In this section, we will solve the basic equation (\ref{phi}) in a nonconventional way. Firstly, let us introduce two special transformations for the function $\p$:
\begin{eqnarray*}
g_u(\p(s)):=\sqrt{1+us^2}\p\left(\f{s}{\sqrt{1+us^2}}\right),\qquad h_v(\p(s)):=\p(vs),
\end{eqnarray*}
where $u$ and $v$ are constants with $v\neq0$. The motivation of above transformations can be found in \cite{yct-dhfp}, here we just need to know that such transformations satisfy
\begin{eqnarray*}
g_{u_1}\circ g_{u_2}=g_{u_1+u_2},\qquad h_{v_1}\circ h_{v_2}=h_{v_1v_2},\qquad h_v\circ g_u=g_{v^2u}\circ h_v,
\end{eqnarray*}
and hence generate a transformation group $G$ under the above generation relationship, which is isomorphism to $\left(\mathbb R\times\mathbb R\backslash\{0\},\cdot\right)$ under the map
$\pi(g_u\circ h_v)=(u,v)$.
For the later, the operation is given by $(u_1,v_1)\cdot(u_2,v_2)=(u_1+v_1^2u_2,v_1v_2)$. In particular,
$$g_u^{-1}=g_{-u},\quad h_v^{-1}=h_{v^{-1}}.$$

The importance of the transformation group $G$ for our question is that the solution space of the 3-parameters equation (\ref{phi}) is invariant under the action of $G$ as below. The computations are  elementary and hence omitted here.
\begin{lemma}\label{l1}
If $\p(s)$ satisfies (\ref{phi}), then the function $\psi(s):=g_u(\p)$ satisfies the same kind of equation
$$s(k'_2-k'_3s^2)(\psi\psi'-s\psi'^2-s\psi\psi'')-(\psi'^2+\psi\psi'')+k'_1\psi(\psi-s\psi')=0,$$
where
$$k_1'=k_1+u,~k_2'=k_2+2u,~k_3'=k_3-k_2u-u^2.$$
Moreover, $\p(0)=\psi(0)$ and $\p'(0)=\psi'(0)$.
\end{lemma}

\begin{lemma}\label{l2}
If $\p(s)$ satisfies (\ref{phi}), then the function $\varphi(s):=h_v(\p)$ satisfies the same kind of equation
$$s(k''_2-k''_3s^2)(\varphi\varphi'-s\varphi'^2-s\varphi\varphi'')
-(\varphi'^2+\varphi\varphi'')+k''_1\varphi(\varphi-s\varphi')=0,$$
where
$$k_1''=v^2k_1,~k_2''=v^2k_2,~k_3''=v^4k_3.$$
Moreover, $\p(0)=\varphi(0)$ and $\p'(0)=v\varphi'(0)$.
\end{lemma}

Further more, there are some invariants. Denote
$$\Delta_1=k_2^2+4k_3,\qquad\Delta_2=k_2-2k_1,\qquad\Delta_3=k_1^2-k_1k_2-k_3.$$
Then we have
\begin{lemma}
$\mbox{Sgn}(\Delta_i)~(i=1,2,3)$ are all invariants under the action of $G$.
\end{lemma}
\begin{proof}
It's only need to show that $\mbox{Sgn}(\Delta_i)$~are invariant for $g_u(\p)$ and $h_v(\p)$. It is obvious, because by Lemma \ref{l1} and Lemma \ref{l2} we have $\Delta_1'=\Delta_1$, $\Delta_2'=\Delta_2$, $\Delta_3'=\Delta_3$ and $\Delta_1''=v^4\Delta_1$, $\Delta_2''=v^2\Delta_2$, $\Delta_3''=v^4\Delta_3$.
\end{proof}
Further more, $\Delta_i$ satisfy $\Delta_2^2-4\Delta_3=\Delta_1$. They will play a basic role for the further research.

Next, we will solve the equation (\ref{phi}) with the initial conditions
$$\p(0)=1,\qquad\p'(0)=\varepsilon$$
combining with the transformation group $G$. Note that for $\ab$-metrics $F=\pab$, the function $\ps$ must be positive near $s=0$ and hence we can always assume $\p(0)=1$ after necessary scaling. On the other hand, $\varepsilon\neq0$ by the assumption of Theorem \ref{maincf}.

Let $\psi(s)=g_{-k_1}(\p)$. According to Lemma \ref{l1}, the function $\psi(s)$ will satisfies the following equation
\begin{eqnarray}\label{eq2}
s\{k_2-2k_1-(k_3+k_1k_2-k_1^2)s^2\}(\psi\psi'-s\psi'^2-s\psi\psi'')
-\psi'^2+\psi\psi''=0
\end{eqnarray}
with the initial conditions
$$\psi(0)=1,\qquad\psi'(0)=\varepsilon.$$
Let $u(s)=\psi^2(s)$. It is easy to see that (\ref{eq2}) becomes
\begin{eqnarray}\label{eq3}
\{1+\Delta_2s^2+\Delta_3s^4\}u''=s\{\Delta_2+\Delta_3s^2\}u'
\end{eqnarray}
with the initial conditions
$$u(0)=1,\qquad u'(0)=2\epsilon.$$
Hence, $u'(s)$ is given by
$$u'(s)=\exp\left\{\frac{1}{2}\int\f{\Delta_2+\Delta_3s^2}{1+\Delta_2s^2+\Delta_3s^4}\ud s^2\right\}:=2\varepsilon f(s),$$
where $f(s)$ satisfying $f(0)=1$ can be expressed as elementary functions. So we have
\begin{lemma}
The solutions of equation (\ref{eq3}) with the initial conditions $u(0)=1,~u'(0)=2\epsilon$ are given by $$u(s)=1+2\epsilon\int_0^sf(\sigma)\ud\sigma,$$
where $f(s)$ satisfying $f(0)=1$ are given in the following:
\begin{enumerate}
\item when $\Delta_3=0,~\Delta_1=0$,
$$f(s)=1;$$
\item when $\Delta_3=0,~\Delta_1\neq0$,
$$f(s)=\sqrt{1+\Delta_2s^2};$$
\item when $\Delta_3\neq0,~\Delta_1>0$,
$$f(s)=\sqrt[4]{1+\Delta_2s^2+\Delta_3s^4}\left\{\frac{2+(\Delta_2+\sqrt{\Delta_1})s^2}
{2+(\Delta_2-\sqrt{\Delta_1})s^2}\right\}^\frac{\Delta_2}{4\sqrt{\Delta_1}};$$
\item when $\Delta_3\neq0,~\Delta_1=0$,
$$f(s)=\sqrt{1+\frac{\Delta_2}{2}s^2}\exp\left\{\frac{1}{2+\Delta_2s^2}-\frac{1}{2}\right\};$$
\item when $\Delta_3\neq0,~\Delta_1<0$,
$$f(s)=\sqrt[4]{1+\Delta_2s^2+\Delta_3s^4}\exp\left\{\frac{\Delta_2}{2\sqrt{-\Delta_1}}
\left[\arctan\frac{\Delta_2+2\Delta_3s^2}{\sqrt{-\Delta_1}}-\arctan\frac{\Delta_2}{\sqrt{-\Delta_1}}
\right]\right\}.$$
\end{enumerate}
\end{lemma}

\begin{theorem}
The solutions of equation (\ref{phi}) with the initial conditions $\phi(0)=1,~\phi'(0)=\epsilon$ are given by
$$\phi(s)=\sqrt{(1+k_1s^2)\left\{1+2\epsilon\int_0^s(1+k_1\sigma^2)^{-\frac{3}{2}}
f(\frac{\sigma}{\sqrt{1+k_1\sigma^2}})\ud\sigma\right\}}.$$
\end{theorem}
\begin{proof}
By assumption,
$$\psi(s)=\sqrt{u}=\sqrt{1+2\varepsilon\int_0^sf(\sigma)\ud\sigma},$$
so
\begin{eqnarray*}
\phi(s)&=&g_{k_1}(\psi)\\
&=&\sqrt{1+k_1s^2}\psi(\f{s}{\sqrt{1+k_1s^2}})\\
&=&\sqrt{(1+k_1s^2)\left(1+2\varepsilon\int_0^\frac{s}{\sqrt{1+k_1s^2}}f(\sigma)\ud\sigma\right)},
\end{eqnarray*}
which can also be expressed as the form given in the Theorem.
\end{proof}

Most of the solutions of (\ref{phi}) are non-elementary. Some elementary solutions are listed below~(except for the last two items). Notice that there is no sum of formula when the sum index $n=1$, and we rule $m!!=1$ when $m\leq0$.
\begin{itemize}
\item When $k_1=0,k_2=0,k_3=0$,
$$\ps=\sqrt{1+2\epsilon s};$$
\item When $k_1=0,k_2<0,k_3=0$,
$$\ps=\sqrt{1+\epsilon\left(s\sqrt{1+k_2s^2}+\frac{1}{\sqrt{-k_2}}\arcsin\sqrt{-k_2}s\right)};$$
\item When $k_1=0,k_2>0,k_3=0$,
$$\ps=\sqrt{1+\epsilon\left(s\sqrt{1+k_2s^2}+\frac{1}{\sqrt{k_2}}\textrm{arcsinh}\,\sqrt{k_2}s\right)};$$
\item When $k_3=0,k_1+k_2=0$,
$$\ps=\sqrt{1+2\epsilon s+k_1s^2};$$
\item When $k_1\neq0,k_2=\frac{1}{2n}k_1~(n=1,2,3,\cdots),k_3=0$,
$$\ps=\sqrt{1+k_1s^2+\epsilon s\sqrt{1+k_2s^2}\left[\f{(2n)!!}{(2n-1)!!}
-\sum_{k=1}^{n-1}\f{2(n-k)(2n-2)!!(2k-3)!!}{(2n-1)!!(2k)!!}(1+k_2s^2)^{-k}\right]};$$
\item When $k_1>0,k_2=\frac{1}{2n+1}k_1~(n=1,2,3,\cdots),k_3=0$,
\begin{eqnarray*}
\ps&=&\Bigg\{(1+k_1s^2)\left[1+\frac{(2n-1)!!}{(2n)!!}\frac{\epsilon}{\sqrt{k_2}}\arctan\sqrt{k_2}s\right]\\
&&+\epsilon s\left[\frac{(2n+1)!!}{(2n)!!}-\sum_{k=1}^{n-1}
\frac{2(n-k)(2n-1)!!(2k-2)!!}{(2n)!!(2k+1)!!}(1+k_2s^2)^{-k}\right]\Bigg\}^\frac{1}{2};
\end{eqnarray*}
\item When $k_1<0,k_2=\frac{1}{2n+1}k_1~(n=1,2,3,\cdots),k_3=0$,
\begin{eqnarray*}
\ps&=&\Bigg\{(1+k_1s^2)\left[1+\frac{(2n-1)!!}{(2n)!!}\frac{\epsilon}{\sqrt{-k_2}}
\textrm{arctanh}\,\sqrt{-k_2}s\right]\\
&&+\epsilon s\left[\frac{(2n+1)!!}{(2n)!!}-\sum_{k=1}^{n-1}
\frac{2(n-k)(2n-1)!!(2k-2)!!}{(2n)!!(2k+1)!!}(1+k_2s^2)^{-k}\right]\Bigg\}^\frac{1}{2};
\end{eqnarray*}
\item When $k_1\neq0,k_2=-\frac{1}{2n+1}k_1~(n=1,2,3,\cdots),k_3=0$,
$$\ps=\sqrt{1+k_1s^2+\epsilon s\left[\f{(2n+2)!!}{(2n+1)!!}
-\sum_{k=1}^{n}\f{2(n-k+1)(2n)!!(2k-3)!!}{(2n+1)!!(2k)!!}(1+k_2s^2)^{k}\right]};$$
\item When $k_1>0,k_2=-\frac{1}{2n}k_1~(n=1,2,3,\cdots),k_3=0$,
\begin{eqnarray*}
\ps&=&\Bigg\{(1+k_1s^2)\left[1+\frac{(2n-1)!!}{(2n)!!}\frac{\epsilon}{\sqrt{-k_2}}
\arcsin\sqrt{-k_2}s\right]\\
&&+\epsilon s\sqrt{1+k_2s^2}\left[\frac{(2n+1)!!}{(2n)!!}-\sum_{k=1}^{n-1}
\frac{2(n-k)(2n-1)!!(2k-2)!!}{(2n)!!(2k+1)!!}(1+k_2s^2)^{k}\right]\Bigg\}^\frac{1}{2};
\end{eqnarray*}
\item When $k_1<0,k_2=-\frac{1}{2n}k_1~(n=1,2,3,\cdots),k_3=0$,
\begin{eqnarray*}
\ps&=&\Bigg\{(1+k_1s^2)\left[1+\frac{(2n-1)!!}{(2n)!!}\frac{\epsilon}{\sqrt{k_2}}
\textrm{arcsinh}\,\sqrt{k_2}s\right]\\
&&+\epsilon s\sqrt{1+k_2s^2}\left[\frac{(2n+1)!!}{(2n)!!}-\sum_{k=1}^{n-1}
\frac{2(n-k)(2n-1)!!(2k-2)!!}{(2n)!!(2k+1)!!}(1+k_2s^2)^{k}\right]\Bigg\}^\frac{1}{2};
\end{eqnarray*}
\item When $k_1=0,k_2=0,k_3\neq0$,
$$\ps=\sqrt{1+2\epsilon\int_0^s\sqrt[4]{1-k_3\sigma^4}\ud\sigma};$$
\item When $k_1\neq0,k_2=0,k_3=0$,
$$\ps=\sqrt{(1+k_1s^2)\left[1+2\epsilon\int_0^s\f{e^{\frac{k_1}{2}\sigma^2}}{(1+k_1\sigma^2)^2}
\ud\sigma\right]}.$$
\end{itemize}

\section{Some explicit examples}\label{5}

We can construct some typical examples below.

\begin{example}
Take $k_1=k_2=k_3=0$ and $\varepsilon=\frac{1}{2}$, then $\ps=\sqrt{1+s}$ satisfies (\ref{phi}). By Theorem \ref{main1}, the Finsler metric
$$F=\sqrt{\a(\a+\b)}$$
is locally dually flat if and only if $\a$ is locally dually flat and $\b$ is dually related to $\a$. In particular, the following metrics
$$F=\sqrt{\frac{\sqrt{(1+\mu|x|^2)|y|^2-\mu\langle x,y\rangle^2}}{(1+\mu|x|^2)^\frac{3}{4}}\left(\frac{\sqrt{(1+\mu|x|^2)|y|^2-\mu\langle x,y\rangle^2}}{(1+\mu|x|^2)^\frac{3}{4}}+\f{\lambda\xy}{(1+\mu|x|^2)^\frac{5}{4}}\right)}$$
are dually flat.
\end{example}

\begin{example}
Take $k_1=-k_2=\kappa$, $k_3=0$, then $\ps=\sqrt{1+2\varepsilon s+\kappa s^2}$ satisfies (\ref{phi}). By Theorem \ref{main1}, the Finsler metric
$$F=\sqrt{\alpha^2+2\varepsilon\a\b+\kappa\b^2}$$
is locally dually flat if and only if
$$\a=(1-\kappa\bar b^2)^{-1}\sqrt{(1-\kappa\bar b^2)\ba^2+\kappa\bb^2},\qquad\b=-(1-\kappa\bar b^2)^{-1}\bb,$$
where $\ba$ is locally dually flat and $\bb$ is dually related to $\ba$.
\end{example}

\begin{example}
Take $k_1=k_3=0,$ $k_2=-1$ and $\varepsilon=1$, then $\ps=\sqrt{1+s\sqrt{1-s^2}+\arcsin s}$ satisfies (\ref{phi}). By Theorem \ref{main1}, the Finsler metric
$$F=\sqrt{\a^2+\sqrt{\a^2-\b^2}\b+\a^2\arcsin\f{\b}{\a}}$$
is locally dually flat if and only if
$$\a=(1-\bar b^2)^{-\frac{3}{4}}\sqrt{(1-\bar b^2)\ba^2+\bb^2},\qquad\b=-(1-\bar b^2)^{-\frac{3}{4}}\bb,$$
where $\ba$ is locally dually flat and $\bb$ is dually related to $\ba$.
\end{example}

\begin{example}
Take $k_1=k_3=0,$ $k_2=1$ and $\varepsilon=1$, then $\ps=\sqrt{1+s\sqrt{1+s^2}+\mathrm{arcsinh}\,s}$ satisfies (\ref{phi}). By Theorem \ref{main1}, the Finsler metric
$$F=\sqrt{\a^2+\sqrt{\a^2+\b^2}\b+\a^2\mathrm{arcsinh}\f{\b}{\a}}$$
is locally dually flat if and only if
$$\a=(1+\bar b^2)^{-\frac{3}{4}}\sqrt{(1+\bar b^2)\ba^2-\bb^2},\qquad\b=-(1+\bar b^2)^{-\frac{3}{4}}\bb,$$
where $\ba$ is locally dually flat and $\bb$ is dually related to $\ba$.
\end{example}

\begin{example}
Take $k_1=k_2=0,$ $k_3=\pm1$ and $\varepsilon=\frac{1}{2}$, then $\ps=\sqrt{1+\int_0^s\sqrt[4]{1\pm\sigma^4}\ud\sigma}$ satisfies (\ref{phi}). By Theorem \ref{main1}, the Finsler metric
$$F=\sqrt{1+\int_0^\frac{\b}{\a}\sqrt[4]{1\pm\sigma^4}\ud\sigma}$$
is locally dually flat if and only if
$$\a=(1\mp\bar b^4)^{-\frac{5}{8}}\sqrt{(1\mp\bar b^4)\ba^2\pm\bar b^2\bb^2},\qquad\b=-(1\mp\bar b^4)^{-\frac{5}{8}}\bb,$$
where $\ba$ is locally dually flat and $\bb$ is dually related to $\ba$.
\end{example}

\begin{example}
Take $k_2=k_3=0,$ $k_1=\pm1$ and $\varepsilon=\frac{1}{2}$, then $\ps=\sqrt{(1\pm s^2)(1+\int_0^s\f{e^{\pm\frac{\sigma^2}{2}}}{(1\pm\sigma^2)^2}\ud\sigma)}$ satisfies (\ref{phi}). By Theorem \ref{main1}, the Finsler metric
$$F=\sqrt{(\a^2\pm\b^2)\left(1+\int_0^\frac{\b}{\a}\f{e^{\pm\frac{\sigma^2}{2}}}{(1\pm\sigma^2)^2}\ud\sigma\right)}$$
is locally dually flat if and only if
$$\a=e^{\pm\frac{\bar b^2}{4}}\ba,\qquad\b=-e^{\pm\frac{\bar b^2}{4}}\bb,$$
where $\ba$ is locally dually flat and $\bb$ is dually related to $\ba$.
\end{example}

\noindent Changtao Yu\\
School of Mathematical Sciences, South China Normal
University, Guangzhou, 510631, P.R. China\\
aizhenli@gmail.com

\begin{thebibliography}{00}
\bibitem{AN1}
S.-I. Amari, {\em Differential-Geometrical methods in Statistics}, Springer Lecture Notes in Statistics, {\bf 28}, Springer-Verlag, 1985.
\bibitem{AN2}
S.-I. Amari and H. Nagaoka, {\em Methods of information geometry}, AMS Translation of Math. Monographs, {\bf 191}, Oxford University Press, 2000.
\bibitem{bcs}
S. B\'{a}cs\'{a}, X. Cheng and Z. Shen, {\em Curvature properties of $\ab$-metrics}, In ``Finsler Geometry, Sapporo 2005-In Memory of
Makoto Matsumoto", ed. S. Sabau and H. Shimada, Advanced Studies in
Pure Mathematics {\bf 48}, Mathematical Society of Japan, 2007, 73-110.
\bibitem{db-robl-onri}
D. Bao and C. Robles, {\em On Ricci curvature and flag curvature in Finsler geometry}, in "{\em A Sampler of Finsler Geometry}" MSRI series, Cambridge University Press, 2004.
\bibitem{brs}
D. Bao, C. Robles and Z. Shen, {\em Zermelo navigation on Riemannian manifolds}, J. Diff. Geom. {\bf 66} (2004), 391-449.
\bibitem{csz-oldf}
X. Cheng, Z. Shen and Y. Zhou, {\em On locally dually flat Randers metrics}, Intern. Math., {\bf 21} (2010), 1531-1543.
\bibitem{cxy-szm-tyf}
X. Cheng, Z. Shen and Y. Tian, {\em Einstein $(\alpha,\beta)$-metrics}, Israel J. Math. to appear.
\bibitem{huang}
X. Mo and L. Huang, {\em On curvature decreasing property of a class of navigation problems}, Publ. Math. Debrecen {\bf 71}, (2007), 141-163.
\bibitem{lb-szm-onac}
B. Li and Z. Shen, {\em On a class of projectively flat Finsler metrics with constant flag curvature}, Int. J. Math. {\bf 18} (2007), 1-12.
\bibitem{szm-opfa}
Z. Shen, {\em On projectively flat $\ab$-metrics}, Can. Math. Bull. {\bf 52} (2009), 132-144.
\bibitem{szm-rfgw}
Z. Shen, {\em Riemann-Finsler geometry with applications to information geometry}, Chinese Ann. Math. Ser. B, {\bf 27}(1) (2006), 73-94.
\bibitem{szm-yct-oesm}
Z. Shen and C. Yu, {\em On Einstein square metrics}, preprint. \url{http://arxiv.org/abs/1209.3876}
\bibitem{xoldf}
Q. Xia, {\em On locally dually flat $\ab$-metrics}, Diff. Geom. Appl., {\bf 29}, (2011), 233-243.
\bibitem{yct-dhfp}
C. Yu, {\em Deformations and Hilbert's Fourth Problem}, preprint. \url{http://arxiv.org/abs/1209.0845}
\bibitem{yct-odfr}
C. Yu, {\em On dually flat Randers metrics}, preprint. \url{http://arxiv.org/abs/1209.1150}
\bibitem{zhou}
L. Zhou, {\em A local classication of a class of $\ab$-metrics with constant flag curvature}, Diff. Geom. Appl. {\bf 28} (2010), 170-193.
\end{thebibliography}
\end{document}